\documentclass[12 pt, reqno]{amsart}

\usepackage{amsthm,amssymb,amstext,amscd,amsfonts,amsbsy,amsxtra,latexsym,cite}
\usepackage{mathrsfs}
\usepackage{verbatim}
\usepackage[latin1]{inputenc}
\usepackage{graphicx, graphics, wrapfig}

\newcommand{\field}[1]{\mathbb{#1}}

\newcommand{\N}{\field{N}}
\newcommand{\Z}{\field{Z}}

\newcommand{\C}{\field{C}}
\def\({\left(}
\def\){\right)}

\newcommand{\sB}{\mathscr{B}}
\newcommand{\sC}{\mathscr{C}}
\newcommand{\sE}{\mathscr{E}}
\newcommand{\im}{\text{Im}}
\newcommand{\re}{\text{Re}}
\newcommand{\calC}{\mathcal{C}}
\newcommand{\calH}{\mathcal{H}}
\newcommand{\calU}{\mathcal{U}}
\newcommand{\calV}{\mathcal{V}}
\newcommand{\Prob}{\mathbf{P}}
\renewcommand{\o}{\overline}

\renewcommand{\o}{\overline}

\theoremstyle{plain}
\newtheorem{theorem}{Theorem}
\newtheorem*{theorem*}{Theorem}
\newtheorem{lemma}[theorem]{Lemma}
\newtheorem{proposition}[theorem]{Proposition}
\newtheorem{corollary}[theorem]{Corollary}
\newtheorem*{conjecture*}{Conjecture}

\theoremstyle{definition}

\theoremstyle{remark}

\newtheorem*{remark*}{Remark}
\newtheorem*{remarks}{Remarks}

\numberwithin{theorem}{section} \numberwithin{equation}{section}

\begin{document}

\title{Schur's partition theorem and mixed mock modular forms}
\author{Kathrin Bringmann}
\address{Mathematical Institute\\University of
Cologne\\ Weyertal 86-90 \\ 50931 Cologne \\Germany}
\email{kbringma@math.uni-koeln.de}
\author{Karl Mahlburg}
\address{Department of Mathematics \\
Louisiana State University \\
Baton Rouge, LA 70802\\ U.S.A.}
\email{mahlburg@math.lsu.edu}


\date{\today}


\thanks{The research of the first author was supported by the Alfried Krupp Prize for
Young University Teachers of the Krupp Foundation and an individual research grant from the ERC.  The second author was supported by NSF Grant DMS-1201435.}

\begin{abstract}

We study families of partitions with gap conditions that were introduced by Schur and Andrews, and describe their fundamental connections to combinatorial $q$-series and automorphic forms.  In particular, we show that the generating functions for these families naturally lead to deep identities for theta functions and Hickerson's universal mock theta function, which provides a very general answer to Andrews' Conjecture on the modularity of the Schur-type generating function.  Furthermore, we also complete the second part of Andrews' speculation by determining the asymptotic behavior of these functions.  In particular, we use Wright's Circle Method in order to prove families of asymptotic inequalities in the spirit of the Alder-Andrews Conjecture.  As a final application, we prove the striking result that the universal mock theta function can be expressed as a conditional probability in a certain natural probability space with an infinite sequence of independent events.
\end{abstract}

\maketitle

\section{Introduction and statement of results}

The famous Rogers-Ramanujan identities show the equality of a hypergeometric $q$-series and an infinite product.  If $r = 1 \text{ or } 2$, then the two identities can be simultaneously stated as \cite{RR19}
\begin{equation}
\label{E:RR}
\sum_{n\geq 0} \frac{q^{n^2 + (r-1)n}}{(q;q)_n} = \frac{1}{(q^{r}; q^5)_\infty (q^{5-r}; q^5)_\infty}.
\end{equation}
Throughout the paper we use for $n\in\N_0 \cup\{\infty\}$ the standard $q$-factorial notation $(a)_n = (a;q)_n := \prod_{j = 0}^{n-1} (1-aq^j),$ as well as the the additional shorthand $(a_1, \dots , a_r)_n := (a_1)_n \cdot \dots \cdot (a_r)_n$.

The Rogers-Ramanujan identities have had a tremendous influence throughout mathematics in the more than one hundred years since they were first discovered.  Generalizations and applications of the identities have inspired developments in combinatorial and analytic partition theory \cite{Gor61, Stem90}; the theory of infinite continued fractions \cite{AG93,Gor65}; the theory of symmetries and transformations for hypergeometric $q$-series \cite{And66, And75}; the exact solution of the hard hexagon model in statistical mechanics \cite{And81}; and vertex operator algebras \cite{LW81}.  Here we study yet another direction, as we focus on the role of identities such as \eqref{E:RR} in the theory of automorphic forms.

In general it is a very challenging problem to determine the automorphic properties of a hypergeometric $q$-series (for example, see the discussion of Nahm's Conjecture in Section II.3 of \cite{Zag06}).  From this perspective, the Rogers-Ramanujan identities are nothing short of incredible, as they equate a hypergeometric series written in ``Eulerian'' form on the left-hand side to an infinite product that is recognizable as a simple modular function on the right.  In this paper we consider families of identities related to \eqref{E:RR} whose automorphic properties have not been previously determined, and we identify the surprisingly simple automorphic forms that underlie the $q$-series.

Before we begin to describe our results, we note that it is helpful to understand the Rogers-Ramanujan identities as combinatorial identities for integer partitions with gap or congruential conditions.  If $\lambda$ is a partition of $n$, then $\lambda$ consists of parts $\lambda_1 \geq \dots \geq \lambda_k \geq 1$ that sum to $n$; in this case we write $\lambda \vdash n$ (see the reference \cite{And98} for additional standard notation and terminology).  In particular, let $B_{1}(n)$ denote the number of partitions of $n$ such that each pair of parts differs by at least $2$, and let $C_{1}(n)$ count the number of such partitions where the smallest part is also at least $2$. Furthermore, if $d \geq 3$ and $1 \leq r \leq \frac{d}{2}$, then we let $D_{d,r}(n)$ denote the number of partitions of $n$ into parts congruent to $\pm r \pmod{d}$.  The Rogers-Ramanujan identities \eqref{E:RR} are then equivalent to the combinatorial statements that
\begin{equation*}
B_1(n) = D_{5,1}(n) \qquad \text{and} \qquad C_1(n) = D_{5,2}(n).
\end{equation*}

Following Rogers-Ramanujan, the next major development in the subject was due to Schur \cite{Sch26}, who proved a similar identity for partitions with parts differing by at least $3$.   In fact, Gleissberg \cite{Gle28} extended Schur's result to a general modulus, which we state in full below. Let $B_{d,r}(n)$ denote the number of partitions of $n$ such that each part is congruent to $0, \pm r \pmod{d}$, each pair of parts differs by at least $d$, and if $d \mid \lambda_i$, then $\lambda_i - \lambda_{i+1} > d.$  We denote the generating function by
\begin{equation}
\label{E:Bdrq}
\sB_{d,r}(q) := \sum_{n \geq 0} B_{d,r}(n) q^n.
\end{equation}
Furthermore, let $\sE_{d,r}(q)$ denote the generating function for partitions into distinct parts that are congruent to $\pm r \pmod{d}$, with enumeration function $E_{d,r}(n)$, so that
\begin{equation}
\label{E:Edrq}
\sE_{d,r}(q) := \sum_{n \geq 0} E_{d,r}(n) q^n = \prod_{n \geq 0} \left(1 + q^{r + dn}\right) \left(1 + q^{d-r + dn}\right) = \left(-q^r, -q^{d-r}; q^d\right)_\infty.
\end{equation}
Schur's general identity is then stated as follows.
\begin{theorem*}[\cite{Gle28,Sch26}]
If $d \geq 3$, $1 \leq r < \frac{d}{2}$, then
\begin{equation}
\label{E:Schur}
\sB_{d,r}(q) = \sE_{d,r}(q).
\end{equation}
\end{theorem*}
\begin{remarks}
{\it 1.} The case $d = 3$ and $r=1$ is the most often cited case of Schur's identities, as it is easily seen that $E_{3,1}(n) = D_{6,1}(n)$, and the resulting restatement of the theorem, the identity $B_{3,1}(n) = D_{6,1}(n)$, is directly analogous to the case $r=1$ of \eqref{E:RR}.  Indeed, the first Rogers-Ramanujan identity may be viewed as a degenerate case of Schur's Theorem, stating $B_{1,1}(n) = D_{5,1}(n)$.

{\it 2.} Euler's Theorem states that the number of partitions of $n$ into distinct parts is equal to those with odd parts (Cor. 1.2 in \cite{And98}).  This can also be viewed as a degenerate Schur-type result, as it states that the partitions of $n$ with gaps of at least $1$ are equinumerous to $D_{4,1}(n)$.

{\it 3.} One may assume in Schur's Theorem that the greatest common divisor is $(d,r)=1$, as if $(d,r) = g > 1$, then the statement reduces to the smaller case of $d' = \frac{d}{g}$ and $r' = \frac{r}{g}.$

{\it 4.} Bressoud \cite{Bre80} and Alladi and Gordon \cite{AG93} later gave bijective proofs of Schur's identities that also provide additional combinatorial information regarding the distribution of parts.
\end{remarks}

Schur's family of identities relates partitions with gap conditions (with generating functions $\sB_{d,r}(q)$) to infinite products that are essentially modular forms (as in \eqref{E:Bdrq}), and can in particular be expressed as a simple quotient of theta functions (this will be made precise later in the paper).  In contrast, in \cite{And68} Andrews considered a partition function related to the second of the Rogers-Ramanujan identities, which resulted in a $q$-series with more exotic automorphic behavior.

In general, we let $C_{d,r}(n)$ count the number of partitions enumerated by $B_{d,r}(n)$ that also satisfy the additional restriction that the smallest part is larger than $d$, and define the corresponding generating function as
\begin{equation*}
\sC_{d,r}(q) := \sum_{n \geq 0} C_{d,r}(n) q^n.
\end{equation*}
Andrews provided the following evaluation of the generating function for these ``Schur-type'' partitions for the case of $d=3$, $r=1$.
\begin{theorem*}[Andrews \cite{And68}]
We have that
\begin{equation}
\label{E:AndrewsC3}
\sC_{3,1}(q) = \frac{(-q;q)_\infty}{\left(q^6; q^6\right)_\infty}
\sum_{n \geq 0} \frac{(-1)^n q^{\frac{9n(n+1)}{2}}(1-q^{6n+3})}{(1+q^{3n+1})(1+q^{3n+2})}.
\end{equation}
\end{theorem*}
\noindent After stating this result, Andrews then commented that
\begin{quote}
{\it \dots the generating function for $C_{3,1}(n)$ is similar to the mock theta functions.  Indeed, it is conceivable that a very accurate asymptotic formula for $C_{3,1}(n)$ may be found\dots.}
\end{quote}

We not only answer Andrews' claim about $\sC_{3,1}(q)$, but achieve much more - we fully describe the automorphic properties of all the ``Schur-type'' generating functions.  This description is only possible due to Zwegers' groundbreaking thesis \cite{Zw02}, where he thoroughly described how Ramanujan's famous mock theta functions (see \cite{Wat36}) fit into the modern framework of real-analytic automorphic forms as developed in \cite{BruF04}.
We adopt the terminology from \cite{DMZ12} and \cite{Zag09} when discussing $q$-series that have automorphic properties similar to Ramanujan's mock theta functions (see Section \ref{S:Gen:Auto} for definitions).  In particular, a {\it (weak) mixed mock modular form} is a function that lies in the tensor product of the general spaces of mock modular forms and (weakly holomorphic) modular forms, possibly multiplied with an additional rational multiple of $q$.  The following result shows that the Schur-type generating functions are examples of such forms (see \cite{ARZ13,BM11,BO09,CDH13,DMZ12}
for many other applications of mixed mock modular forms).

\begin{theorem}
\label{T:mixedmock}
If $d \geq 3$, $1 \leq r < \frac{d}{2}$, then $\sC_{d,r}(q)$ is a mixed mock modular form.
\end{theorem}
\begin{remark*}
Similarly, it was already well-known to experts that $\sB_{d,r}(q)$ is a modular function (up to a $q$-power); this is apparent from combining \eqref{E:Edrq} and \eqref{E:Schur} with the general theory of theta functions and modular units (cf. Proposition \ref{P:Bmod}).
\end{remark*}

In fact, we can precisely describe the automorphic functions that arise in Theorem \ref{T:mixedmock} in terms of fundamental representatives from the spaces of modular and mock modular forms.
In order to do so, we recall Hickerson's {\it universal mock theta function} (of odd order), which is defined by
\begin{equation}
\label{E:Univg3}
g_3(x;q) := \sum_{n \geq 0} \frac{q^{n(n+1)}}{(x; q)_{n+1} (x^{-1}q; q)_{n+1}}.
\end{equation}
Although the correct reference is often misattributed, we hope to make clear that Hickerson was the first to recognize the importance of this function.  Indeed, in \cite{Hic88} he showed that Ramanujan and Watson's examples can be decomposed into expressions in terms of $g_3,$ and thereby proved the so-called ``mock theta conjectures'' by studying the identities satisfied by the universal function.  As we will see later, the sum in the expression $\eqref{E:AndrewsC3}$ is also essentially similar to the Appell-Lerch functions used in Zwegers' study of the mock theta functions \cite{Zw02}; we prove later that \eqref{E:Univg3} can also be expressed in terms of such sums.  The present notation was introduced by Gordon and McIntosh (see their recent survey \cite{GM12}), who also found a second universal mock theta function of ``even order'', which they denoted by $g_2(x;q)$.

In the following theorem statement we adopt standard notation for modular forms and theta functions; the definitions of $\eta(\tau)$ and $\vartheta(w;\tau)$ are reviewed in Section \ref{S:Gen:Auto} (specifically, see \eqref{E:eta} and \eqref{Thetadef}).  Here and throughout the paper we let $q:=e^{2\pi i \tau}$ be the standard uniformizer for the cusp $i\infty$, where $\tau$ is in the complex upper-half plane $\mathbb{H}$.
\begin{theorem}
\label{T:univ}
If $d \geq 3$, $1 \leq r < \frac{d}{2}$, then
\begin{equation*}
\sC_{d,r}(q) = -q^{-\frac{d}{12} +\frac{r}{2}} \frac{\vartheta(\frac{1}{2} + r\tau; d\tau)}{\eta(d\tau)}
g_3\left(-q^r; q^d\right).
\end{equation*}
\end{theorem}
\noindent This result gives a resounding affirmative answer to Andrews' original speculation regarding the relationship between $\sC_{3,1}(q)$ and mock theta functions, as it precisely describes the generating functions as simple mixed mock modular forms.

\begin{remark*}
We will see in the proof that Theorem \ref{T:univ} is also equivalent to the identity
\begin{equation}
\label{E:C=B}
\sC_{d,r}(q) = \sB_{d,r}(q) g_3\left(-q^r; q^d\right).
\end{equation}
Strikingly, this means that the universal mock theta function $g_3(-q^r; q^d)$ plays the role of a combinatorial ``correction factor'' that precisely accounts for the difference in the enumeration functions $B_{d,r}$ and $C_{d,r}$.
\end{remark*}

\begin{remark*}
In Section 8 of Andrews' seminal work on Durfee symbols \cite{And07}, he introduced another combinatorial generating function that is closely related to the universal mock theta function.  In particular, he defined ``odd Durfee symbols'' (which generalize partitions with certain parity conditions) and showed that the two-parameter generating function is
\begin{equation*}
R_{1}^{0}(x;q) := \sum_{n \geq 0} \frac{q^{2n(n+1)+1}}{(xq; q^2)_{n+1} (x^{-1}q; q^2)_{n+1}}.
\end{equation*}
Using \eqref{E:Univg3} and Theorem \ref{T:univ}, we therefore find another notable combinatorial relationship for the Schur-type enumeration functions, namely
\begin{equation*}
\sC_{d,r}(q) = -q^{-\frac{d}{12} +\frac{r-d}{2}} \frac{\vartheta(\frac{1}{2} + r\tau; d\tau)}{\eta(d\tau)}
R_0^1\left(-q^{r-\frac{d}{2}}; q^{\frac{d}{2}}\right).
\end{equation*}
As in \eqref{E:C=B}, we see that the odd Durfee symbol generating function (and an additional $q$-power) is similarly the correction factor between the two Schur-type functions.
\end{remark*}

We also answer the second part of Andrews' grand statement by providing asymptotic formulas for all of the Schur-type partition functions.
\begin{theorem}
\label{T:BCAsymp}
Suppose that $d \geq 3$ and $1 \leq r < \frac{d}{2}$.
As $n \to \infty$
\begin{align*}
B_{d,r}(n) & \sim \frac{1}{2^{\frac{5}{4}}3^{\frac{1}{4}} d^{\frac{1}{4}} n^{\frac{3}{4}}} e^{\pi \sqrt{\frac{2n}{3d}}}, \\
C_{d,r}(n) & \sim \frac{1}{3}\cdot B_{d,r}(n).
\end{align*}
\end{theorem}
\begin{remark*}
In fact, we prove much more than this, as our analysis also allows us to describe further terms in the asymptotic expansion for these enumeration functions (see Theorem \ref{T:Ineq} for further applications).
Moreover, one could also use the extension of the Hardy-Ramanujan Circle Method that the authors developed in \cite{BM11} in order to find expressions for the coefficients with only polynomial error.
\end{remark*}

The results described thus far have followed the spirit of the Rogers-Ramanujan and Schur identities, in which simple enumeration functions are shown to be combinatorially equivalent, and whose generating functions lie in the intersection of hypergeometric $q$-series and automorphic forms.  However, there is also a notable body of research that was inspired by a negative approach to \eqref{E:RR} and \eqref{E:Schur}.  In \cite{Ald48} Alder proved a non-existence result for certain general identities analogous to those of Rogers-Ramanujan and Schur (also see \cite{Ald69}).  Moreover, Andrews observed in \cite{And71P} that
such identities often fail to hold because of an asymptotic inequality in which one enumeration function is eventually always larger than the other.

To precisely describe the cases considered by Alder and Andrews, let $q_{d,j}(n)$ denote the number of partitions of $n$ in which each pair of parts differs by at least $d$ and the smallest part is at least $j$, and let $Q_{d,j}(n)$ denote the number of partitions into parts congruent to $\pm j \pmod{d+3}.$  The Rogers-Ramanujan identities \eqref{E:RR} are then equivalently stated as
\begin{equation*}
q_{2,1}(n) = Q_{2,1}(n) \qquad \text{and} \qquad q_{2,2}(n) = Q_{2,2}(n).
\end{equation*}
Alder showed that if $d \geq 3$, then $q_{d,j}(n)$ is not equal to any partition enumeration function where the parts lie in some specified set of positive integers, taken without any restrictions on multiplicity or gaps.  Furthermore, he conjectured an inequality between $q_{d,j}(n)$ and $Q_{d,j}(n)$ that was proven in a weaker asymptotic version by Andrews \cite{And71P}.  The full Alder-Andrews Conjecture was recently confirmed across a series of papers by Yee \cite{Yee04,Yee08} and Alfes, Jameson, and Lemke Oliver \cite{AJL11}.
\begin{theorem*}
If $d \geq 3$ and $n \geq 2d + 9$, then
\begin{equation*}
q_{d,1}(n) > Q_{d,1}(n).
\end{equation*}
\end{theorem*}

However, although there are now a plethora of identities and non-identities that relate various pairs of the enumeration functions described within this paper, there are essentially no results comparing the enumeration functions in the same family as $d$ and/or $r$ vary.  Our next main result shows that for each fixed $d$, the families of enumeration functions $B_{d,r}$ and $C_{d,r}$ are asymptotically decreasing in $r$.

\begin{theorem}
\label{T:Ineq}
Suppose that $d \geq 3$ and $1 \leq r < \frac{d}{2}$.
\begin{enumerate}
\item
As $n \to \infty$ we have the asymptotic equalities
\begin{align*}
B_{d,r+1}(n) &\sim B_{d,r}(n), \\
C_{d,r+1}(n) &\sim C_{d,r}(n).
\end{align*}
\item
For sufficiently large $n$ we have the inequalities
\begin{align*}
B_{d,r+1}(n) &> B_{d,r}(n), \\
C_{d,r+1}(n) &> C_{d,r}(n).
\end{align*}
\end{enumerate}
\end{theorem}
We can also compare the enumeration families across different $d$ values.
\begin{theorem}
Suppose that $3 \leq d_1 < d_2$ and $1 \leq r_1 < \frac{d_1}{2}$, $1 \leq r_2 < \frac{d_2}{2}.$ For sufficiently large $n$ we have the inequalities
\begin{align*}
B_{d_1, r_1}(n) & > B_{d_2, r_2}(n), \\
C_{d_1, r_1}(n) & > C_{d_2, r_2}(n).
\end{align*}
\end{theorem}
\begin{remark*}
In light of \eqref{E:Schur} and Theorem \ref{T:univ}, the two pairs of asymptotic results from these theorems are statements about the asymptotic expansion of the coefficients of modular forms and mixed mock modular forms, respectively.  In the case of the $B_{d,r}(n)$, it is entirely well-known that exact formulas can be found for the coefficients of modular forms \cite{Rad43}, and we include the inequalities for the sake of completeness.  The second case is more novel, and it is only due to more recent advances in the asymptotic analysis of automorphic $q$-series (as in \cite{BM11}) that we are able to compare the $C_{d,r}(n)$.
\end{remark*}
\begin{remark*}
It would be of great interest if a bijective proof could be found for any of these asymptotic results.
\end{remark*}
Our final results describe the surprising relationship between the Schur-type generating functions and events in certain probability spaces with infinite sequences of independent events.
In particular, we find a remarkable interpretation in terms of conditional probabilities for the universal mock theta function evaluated at real arguments; the precise definitions for the following result are found in Section \ref{S:Prob}.
\begin{theorem}
\label{T:g3Prob}
Suppose that $d \geq 3$, $1 \leq r < \frac{d}{2}$ and $0 < q < 1$ is real.  There are events $W$ and $X$ in a certain probability space (see \eqref{E:UV} and Theorem \ref{T:prob}) such that
\begin{equation*}
\Prob(W \mid X) = g_3\left(-q^r; q^d\right).
\end{equation*}
\end{theorem}
\begin{remark*}
Since probabilities are between $0$ and $1$, Theorem \ref{T:g3Prob} immediately implies that for real $0 \leq q < 1$ we have the striking (and non-obvious) bound
\begin{equation*}
g_3\left(-q^r; q^d\right) < 1.
\end{equation*}
\end{remark*}

The remainder of the paper is structured as follows.  In Section \ref{S:Gen} we carefully study the combinatorics of the Schur-type enumeration functions, deriving a $q$-difference equation whose solution gives useful $q$-series expressions for the generating functions.  We also identify the automorphic properties of these $q$-series.  In Section \ref{S:Asymp} we turn to the asymptotic behavior of the Schur-type functions, using a modification of Wright's Circle Method in order to find asymptotic expansions for the coefficients.  We conclude in Section \ref{S:Prob} by defining a simple probability space that is intimately related to the Schur-type partitions, and use this to prove additional identities for the universal mock theta function.

\section{Generating functions, identities, and automorphic $q$-series}
\label{S:Gen}

In this section we evaluate the generating functions for Schur-type partitions, prove related $q$-series identities, and describe the relationship to automorphic objects such as theta functions and mock theta functions.

\subsection{Generating functions as hypergeometric $q$-series}
We begin by introducing a combinatorial refinement of the enumeration functions, from which we derive a $q$-difference equation; the hypergeometric solution of the equation then gives the desired $q$-series expressions.  Our definitions are influenced by Andrews' work in \cite{And68}, which used the case $d=3$ from the general construction that follows.
Specifically, for integer parameters $d \geq 3, 1 \leq r < \frac{d}{2}$, and $m \geq 1,$ $j, n \geq 0$, we define
\begin{align*}
\beta_{d,r,j}(n, m):=  \# \Big\{ & \lambda \vdash n \: : \: \lambda = \lambda_1 + \dots + \lambda_m \text{ where } \lambda_i > j \text{ and } \lambda_i \equiv 0, \pm r \! \pmod{d} \; \forall i, \\
                    & \text{ with gaps } \lambda_{i}-\lambda_{i+1} \geq d, \text{ and furthermore } \lambda_{i}-\lambda_{i+1} > d \text{ if } d | \lambda_i \Big\}.
\end{align*}
We also adopt the convention that $\beta_{d,r,j}(0,0) := 1.$
It follows immediately from the definitions that
\begin{align}
B_{d,r}(n) & =   \sum_{m \geq 0}\beta_{d,r,0}(m,n), \\
C_{d,r}(n) & =   \sum_{m \geq 0}\beta_{d,r,d}(m,n). \notag
\end{align}

We denote the corresponding generating functions by
\begin{equation*}
f_{d,r}(x) = f_{d,r}(x;q) := \sum_{m,n\geq 0} \beta_{d,r,0}(m,n)x^m q^n.
\end{equation*}
We are then particularly interested in finding hypergeometric series for the cases
\begin{align}
\label{E:BC=f}
\sB_{d,r}(q)= f_{d,r}(1;q) = \sum_{n\geq 0} B_{d,r}(n)q^n, \\
\sC_{d,r}(q)= f_{d,r}(q^d;q) = \sum_{n\geq 0} C_{d,r}(n)q^n. \notag
\end{align}
We achieve this by deriving and then solving the following $q$-difference equation.
\begin{proposition}
\label{P:frec}
For $d \geq 3, 1 \leq r < \frac{d}{2}$,
we have
\begin{align*}
f_{d,r}(x) =  \left(1 + xq^r + xq^{d-r}\right) f_{d,r}\left(xq^d\right) + xq^d\left(1 - xq^d\right) f_{d,r}\left(xq^{2d}\right).
\end{align*}
\end{proposition}
\begin{proof}
We prove the recurrence through a combinatorial inclusion-exclusion argument by conditioning on the smallest part of the partition. Suppose that $\lambda$ is a partition counted by $\beta_{d,r,0}(m,n)$ for some $m$ and $n$.  Then its smallest part is either $r, d-r, d, $ or something larger.  These cases are counted, respectively, by the following sum of generating functions (where for convenience we write $f$ instead of $f_{d,r}$):
\begin{equation}
\label{E:fover}
xq^{r} f\left(xq^d\right) + xq^{d-r} f\left(xq^d\right)
+ xq^{d} f\left(xq^{2d}\right) + 1 \cdot f\left(xq^d\right).
\end{equation}
The term $f\left(xq^d\right)$ ensures that the next smallest part is larger than $d$, while $f\left(xq^{2d}\right)$ gives a part larger than $2d$.

However, \eqref{E:fover} also generates partitions that do not satisfy the Schur-type conditions for $d,r$, so this excess must be subtracted. In particular, if \eqref{E:fover} were precisely to equal to $f(x)$, then iterating the recurrence  would give the term $xq^{d-r} \cdot xq^{d+r} f(xq^{2d}) = x^2 q^{2d} f(xq^{2d})$, which represents a partition with the parts $d-r$ and $d+r$.  Subtracting this unallowed small gap gives the recurrence claimed in the proposition statement.
\end{proof}
\begin{remark*}
Andrews originally proved this result for $d=3$ in \cite{And68} by first describing a family of recurrences satisfied by the $\beta_{3,1,j}(n,m)$, and then turning to their generating functions.  However, it is more direct to instead prove the $q$-difference equation through a direct combinatorial analysis of the underlying partitions, as we have done here.
\end{remark*}
\begin{remark*}
A more general version of the recurrence in Proposition \ref{P:frec} is found in (2.1) of \cite{AG93}, although the most interesting automorphic $q$-series arise from specializing the parameters as in our statement.  However, Alladi and Gordon's study also includes many notable combinatorial results; the reader is particularly encouraged to consult equation (1.1) of \cite{AG93} for further details on the appearance of $f_{d,r}(x;q)$ in the theory of infinite continued fractions.
\end{remark*}

Solving the $q$-difference equation in Proposition \ref{P:frec} gives the following hypergeometric expression for $f_{d,r}.$
\begin{proposition}
\label{P:fqseries}
For $d \geq 3, 1 \leq r < \frac{d}{2}$ and $|q| < 1$, we have
\begin{equation*}
f_{d,r}(x;q) = \left(x;q^d\right)_\infty \sum_{n \geq 0} x^n \frac{\left(-q^r,-q^{d-r};q^d\right)_{n}}{\left(q^d;q^d\right)_{n}}.
\end{equation*}
\end{proposition}
\begin{proof}
It is more convenient to renormalize the recurrence from Proposition \ref{P:frec} by setting
\begin{equation}
\label{E:gdr}
g_{d,r}(x) := \frac{f_{d,r}(x)}{(x;q^d)_{\infty}}.
\end{equation}
Then we have the recurrence (again writing $g$ instead of $g_{d,r}$)
\begin{equation}
\label{E:grec}
(1-x) g(x)= \left(1 + xq^r + x q^{d-r}\right)g\left(x q^d\right) + x q^d g\left(x q^{2d}\right).
\end{equation}
We now consider the expansion of $g$ as a series in $x$, writing
$$
g(x) = \sum_{n \geq 0}A_n x^n,
$$
where $A_n = A_n(q)$ are rational expressions in $q$.

Isolating the coefficient of $x^n$ in \eqref{E:grec} now gives the recurrence
\begin{equation*}
A_n - A_{n-1} = A_n q^{dn} + A_{n-1}\left( q^{d(n-1)+ r} + q^{dn-r} + q^{d(2n-1)}\right).
\end{equation*}
Simplifying, we find that
\begin{equation*}
A_n = \frac{\left(1 + q^{d(n-1)+r}\right)\left(1 + q^{dn-r}\right)}{1-q^{dn}}A_{n-1}.
\end{equation*}
Using the initial condition $A_0 = 1$, we can therefore solve the recurrence to find the unique solution (cf. Lemma 1 in \cite{And68Q})
\begin{equation}
\label{E:ghyper}
g(x) = \sum_{n \geq 0} x^n \frac{\left(-q^r,-q^{d-r};q^d\right)_{n}}{\left(q^d;q^d\right)_{n}}.
\end{equation}
The proof is now complete upon comparison with \eqref{E:gdr}.
\end{proof}

We now use transformations for hypergeometric $q$-series in order to find additional representations for the generating functions that directly display their automorphic properties.
We begin by recalling the following $_{3}\phi_{2}$ transformation, which is equivalent to equation (III.10) in \cite{GR90}
\begin{equation}
\label{E:3phi2}
\sum_{n\geq 0}\frac{\left(\frac{aq}{bc},d,e\right)_{n}}{\left(q,\frac{aq}{b},\frac{aq}{c}\right)_{n}}\left(\frac{aq}{de}\right)^n= \frac{\left(\frac{aq}{d},\frac{aq}{e},\frac{aq}{bc}\right)_{\infty}}{\left(\frac{aq}{b},\frac{aq}{c},\frac{aq}{de}\right)_{\infty}}
\sum_{n\geq 0}\frac{\left(\frac{aq}{de},b,c\right)_{n}}{\left(q,\frac{aq}{d},\frac{aq}{e}\right)_{n}}\left(\frac{aq}{bc}\right)^n.
\end{equation}
We also recall a special case of the Watson-Whipple transformation for $_8\phi_7$ (let $n \to \infty$ in (III.18) of \cite{GR90}), namely
\begin{equation}
\label{E:WW}
\sum_{n\geq 0}\frac{\left(\frac{aq}{bc},d,e\right)_{n}}{\left(q,\frac{aq}{b},\frac{aq}{c}\right)_{n}}\left(\frac{aq}{de}\right)^n = \frac{\left(\frac{aq}{d},\frac{aq}{e}\right)_{\infty}}{\left(aq,\frac{aq}{de}\right)_{\infty}} \sum_{n\geq 0} \frac{\left(a,b,c,d,e\right)_{n}\left(1- a q^{2n}\right)}{\left(q,\frac{aq}{b},\frac{aq}{c},\frac{aq}{d},\frac{aq}{e}\right)_{n}} \frac{\left(aq\right)^{2n}(-1)^n q^{\frac{n(n-1)}{2}}}{(1-a)(bcde)^n}.
\end{equation}

\begin{proposition}
\label{P:Chyper}
For $d \geq 3$ and $1 \leq r < \frac{d}{2}$ we have the identities:
\begin{equation}
\label{P:Chyper:bi}
\sC_{d,r}(q) = \frac{\left(-q^r, -q^{d-r}; q^d\right)_\infty}{\left(q^d; q^d\right)_\infty}\sum_{n\in\Z}\frac{(-1)^n q^{\frac{3dn(n+1)}{2}}}{1+q^{r+dn}}
= \left(-q^{r},-q^{d-r};q^d\right)_{\infty}g_{3}\left(-q^r;q^d\right).
\end{equation}
\end{proposition}
\begin{proof}
The first formula is obtained from \eqref{E:WW} by setting $q \mapsto q^{d}$, $a=x$, $b,c \rightarrow \infty$, $d= -q^{r}$, and $e= -q^{d-r}$.
The left hand side of  \eqref{E:WW} then equals
\begin{equation*}
\sum_{n\geq 0}\frac{\left(-q^r, -q^{d-r}; q^d\right)_n}{\left(q^d; q^d\right)_n} x^n=g_{d,r}(x).
\end{equation*}
The right hand side of \eqref{E:3phi2} simplifies to
\begin{equation*}
\frac{\left(-xq^r,-xq^{d-r}; q^d \right)_{\infty}}{\left(xq^d,x;q^d \right)_{\infty}}\sum_{n \geq 0}\frac{\left(x,-q^r,-q^{d-r};q^d \right)_{n}}{\left(q^d,-xq^r,-xq^{d-r};q^d \right)_{n}} \frac{\left(1 - xq^{2dn}\right)}{1-x}(-1)^n x^{2n} q^{\frac{3dn^2 -dn}{2}}.
\end{equation*}
Multiplying by $(x;q^d)_{\infty}$,
we find that
\begin{equation*}
f_{d,r}(x) =
\frac{\left(-xq^r,-xq^{d-r}; q^d \right)_{\infty}}{\left(xq^d;q^d \right)_{\infty}}\sum_{n \geq 0}\frac{\left(x,-q^r,-q^{d-r};q^d \right)_{n}}{\left(q^d,-xq^r,-xq^{d-r};q^d \right)_{n}} \frac{\left(1 - xq^{2dn}\right)}{1-x}(-1)^n x^{2n} q^{\frac{3dn^2-dn}{2}}.
\end{equation*}
Setting $x=q^d$  gives
\begin{equation}
\label{E:Cdrsubs}
\sC_{d, r}(q)=f_{d,r}\left(q^d\right)=\frac{\left(-q^r,-q^{d-r}; q^d \right)_{\infty}}{\left(q^d;q^d \right)_{\infty}}\sum_{n \geq 0}\frac{(-1)^n \left(1 - q^{d(2n+1)}\right) q^{\frac{3dn^2}{2}+ \frac{3dn}{2}}}{(1+q^{r +dn})(1+ q^{d-r+dn})}.
\end{equation}
Using the routine partial fraction decomposition
\begin{equation*}
\frac{1 - q^{d(2n+1)}}{(1+q^{r +dn})(1+ q^{d-r+dn})} = \frac{1}{1+q^{r+dn}}- \frac{q^{d-r+dn}}{1+ q^{d-r+dn}},
\end{equation*}
we can rewrite the sum in \eqref{E:Cdrsubs} as the bilateral summation
\begin{equation*}
\sum_{n\in\Z}\frac{(-1)^n q^{\frac{3dn(n+1)}{2}}}{1+q^{r+dn}}.
\end{equation*}
This completes the proof of the first identity in \eqref{P:Chyper:bi}.

For the second expression, we note that the left hand side of \eqref{E:3phi2} is the same as in \eqref{E:WW}.  We therefore proceed by making the same substitutions as above: $q \mapsto q^d$, $b,c \rightarrow \infty$, $a=x$, $d=-q^r$, and $e= -q^{d-r}.$
This gives
\begin{equation}
\label{E:gdrhyper2}
g_{d,r}(x)=\frac{\left(-xq^r,-xq^{d-r}; q^d\right)_{\infty}}{\left(x;q^d\right)_{\infty}}\sum_{n \geq 0} \frac{\left(x;q^d\right)_{n}x^n q^{dn^2}}{\left(q^d,-xq^r, -xq^{d-r};q^d\right)_{n}}.
\end{equation}
Setting $x=q^d$ gives the overall expression
\begin{equation*}
\sC_{d, r}(q)=\left(-q^{d+r},-q^{2d-r};q^d\right)_{\infty}\sum_{n \geq 0} \frac{q^{dn(n+1)}}{(-q^{r+d},-q^{2d-r};q^d)_n}.
\end{equation*}
Combined with the definition of the universal mock theta function from \eqref{E:Univg3}, this gives the statement.
\end{proof}
\begin{remark*}
We can also use \eqref{E:3phi2} to recover the formula for $\sB_{d,r}(q)$ from \eqref{E:Schur}.  In particular, if we multiply \eqref{E:gdrhyper2} by $(x; q^d)_\infty$ and set $x=1$,
then every term of the sum vanishes except for $n=0$, and hence we directly obtain the infinite product.
\end{remark*}

\subsection{Automorphic properties}
\label{S:Gen:Auto}
We now describe the automorphicity of the generating functions from above and prove Theorem \ref{T:univ}.  We first recall several basic facts about automorphic and Jacobi forms, although the definitions are stated very roughly (we specifically suppress technical discussion of the finer points of multiplier systems and level structure), as our primary aims are practical; we wish to describe how the combinatorial study of Schur-type partitions leads to the fundamental objects from the theory of automorphic $q$-series, and then to apply the theory to obtain asymptotic results.  The interested reader should consult the cited references for a complete background in the subject.  We also provide special cases of modular transformations for the functions that arise in this paper, as we will use these later for the asymptotic analysis in Section \ref{S:Asymp}.

Briefly, a holomorphic function $f: \mathbb{H} \to \C$ is a {\it weakly holomorphic modular form} of weight $k$ on a congruence subgroup $\Gamma \subset \text{SL}_2(\Z)$ if $f$ is meromorphic at the ``cusps'' of $\Gamma$ and satisfies the modular transformations
\begin{equation}
\label{E:mod}
f\left(\frac{a\tau + b}{c\tau + d}\right) = \chi(\gamma) (c\tau + d)^k f(\tau) \qquad
\text{for all } \gamma = \begin{pmatrix} a & b \\ c& d \end{pmatrix} \in \Gamma,
\end{equation}
where the $\chi(\gamma)$ are certain roots of unity that form a ``multiplier system''.  See \cite{Kob84} for further details.  Furthermore, as in \cite{Zag09}, a holomorphic function $f: \mathbb{H} \to \C$ is a {\it mock modular form} of weight $k$ if there is an associated modular form $g$ of weight $2-k$ (the ``shadow'' of $f$) such that $f + g^\ast$ satisfies a modular transformation of the form \eqref{E:mod}, where the real-analytic correction term is given by
\begin{equation*}
g^\ast(\tau) := \left(\frac{i}{2}\right)^{k-1} \int_{-\overline{\tau}}^\infty \frac{\overline{g\left(-\overline{z}\right)}}{(z+\tau)^k} dz.
\end{equation*}
A {\it mixed mock modular form} is then a product of a modular form and a mock modular form, or a linear combination of such terms; the precise definition reflects a richer tensor structure amongst the vector spaces of automorphic forms, and is found in Section 7.3 of \cite{DMZ12}.
Finally, we also encounter {\it Jacobi forms}, which are complex-valued functions on $\C \times \mathbb{H}$ that satisfy modular-type transformations in the second argument, and certain lattice-invariant transformations in the first argument.  The full theory of such forms may be found in \cite{EZ85}.

We now present the special automorphic functions that arise in the present study.  First, recall Dedekind's eta-function, which is a modular form of weight $\frac{1}{2}$ defined by
\begin{equation}
\label{E:eta}
\eta(\tau) := q^\frac{1}{24}\prod_{n\geq 1} \left(1-q^n\right).
\end{equation}
In particular, it satisfies the inversion formula (Theorem 3.1 in \cite{Apo90})
\begin{equation}
\label{E:etainv}
\eta \left( - \frac{1}{\tau} \right) = \sqrt{-i\tau} \eta(\tau).
\end{equation}
We next recall Jacobi's theta function
\begin{equation}\label{Thetadef}
\vartheta(w) = \vartheta(w;\tau) :=
\sum_{n\in\frac{1}{2}+\Z}e^{\pi i n^2\tau+2\pi in\left(w+\frac{1}{2}\right)},
\end{equation}
which has an equivalent product form (writing $\zeta := e^{2\pi i w}$)
\begin{equation}
\label{E:thetaprod}
\vartheta(w;\tau) = -i q^{\frac{1}{8}} \zeta^{-\frac{1}{2}} \prod_{n\geq 1} (1-q^n)\left(1-\zeta q^{n-1}\right) \left(1-\zeta^{-1}q^n\right).
\end{equation}
This function is a Jacobi form of weight and index $\frac{1}{2}$, and it satisfies the following transformation formulas:
\begin{align}
\label{E:thetaneg}
\vartheta(-w; \tau) & = -\vartheta(w; \tau), \\
\label{E:thetainv}
\vartheta \left( \frac{w}{\tau} ; - \frac{1}{\tau} \right) &= -i \sqrt{-i\tau} e^{\frac{\pi i w^2}{\tau}} \vartheta \left( w; \tau  \right).
\end{align}

Using these definitions the following formula can be immediately verified.
\begin{proposition}
\label{P:Bmod}
If $d \geq 3$, and $1 \leq r < \frac{d}{2}$, then
\begin{equation*}
\sB_{d,r}(q) = -\frac{q^{-\frac{d}{12} +\frac{r}{2}}\vartheta(\frac{1}{2} + r\tau; d\tau)}{\eta(d\tau)}.
\end{equation*}
\end{proposition}

\noindent The statement of Theorem \ref{T:univ} then follows immediately by combining \eqref{E:Schur} and Propositions \ref{P:Chyper} and \ref{P:Bmod}.

We close this section by describing the automorphic properties of $\sC_{d,r}(q)$.
\begin{proof}[Proof of Theorem \ref{T:mixedmock}]
Theorem 3.1 of \cite{Kan09} states that if $\alpha \not \in \Z \tau + \frac{1}{3} \Z$, then $g_3\left(e^{2\pi i \alpha}; \tau\right)$ is a mock modular form of weight $\frac{1}{2}$, up to rational $q$-powers (note that one of the terms in Kang's theorem is the meromorphic function $\frac{\eta^3(3\tau)}{\eta(\tau) \vartheta(3\alpha; 3\tau)},$ which is a weakly holomorphic modular form of weight $\frac{1}{2}$ by Theorem 1.3 in \cite{EZ85}).  Sending $\tau \to d\tau$, we then conclude that $g_3(-q^r; q^d)$ is a mock modular form of weight $\frac{1}{2}$.
\end{proof}

\begin{remark*}
 The theory of mock modular forms (refer to \cite{Zw02}) also provides transformation formulas for $g_3(x;q)$ similar to \eqref{E:thetaneg} and \eqref{E:thetainv}, but this is unnecessary in our asymptotic analysis.
\end{remark*}

\section{Asymptotic results}
\label{S:Asymp}

In this section we determine the asymptotic behavior of the enumeration functions $B_{d,r}(n)$ and $C_{d,r}(n)$, proving Theorems \ref{T:BCAsymp} and \ref{T:Ineq}.  We achieve this by first studying the asymptotic properties of the generating functions, and then applying Wright's version of the Circle Method, which was developed in \cite{Wri41, Wri71} (also see \cite{BM12,BMan13} for further adaptations of the approach in other recent applications).

\subsection{Asymptotic expansions and proof outline}
Our primary goal is to give the first two terms in the asymptotic expansions of the enumeration functions.  Theorems \ref{T:BCAsymp} and \ref{T:Ineq} follow immediately from the following results.  Here $I_s$ denotes the standard modified Bessel function (see Section 4.12 in \cite{AAR99}).
\begin{theorem}
\label{T:BCExpn}
Suppose that $d \geq 3$ and $1 \leq r < \frac{d}{2}.$  As $n\rightarrow\infty$, we have
\begin{align*}
B_{d,r}(n) & =  \alpha_1 n^{-\frac12} I_{-1}\left(\pi\sqrt{\frac{2n}{3d}}\right)
+\beta_1(r) n^{-1} I_{-2}\left(\pi\sqrt{\frac{2n}{3d}}\right)
+O\Big(n^{-\frac32} e^{\frac{\pi\sqrt{2n}}{\sqrt{3d}}}\Big), \\
C_{d,r}(n) & =  \alpha_2 n^{-\frac12} I_{-1}\left(\pi\sqrt{\frac{2n}{3d}}\right)
+ \beta_2(r) n^{-1} I_{-2}\left(\pi\sqrt{\frac{2n}{3d}}\right)
+O\Big(n^{-\frac32} e^{\frac{\pi\sqrt{2n}}{\sqrt{3d}}}\Big),
\end{align*}
where
\begin{eqnarray*}
\begin{array}{ll}
\displaystyle \alpha_1  := \frac{\pi}{\sqrt{6d}},
& \displaystyle \alpha_2  := \frac{\pi}{3\sqrt{6d}} , \\
\displaystyle  \beta_1(r) := \frac{\pi^2}{6d}\left(\frac{d}{12} - \frac{r}{2} + \frac{r^2}{2d}\right),
&
\displaystyle  \beta_2(r) := \frac{\pi^2}{6d}\left(\frac{11d}{108}-\frac{r}{6}+\frac{r^2}{6d}\right).
\end{array}
\end{eqnarray*}
\end{theorem}
\noindent The proof of Theorem \ref{T:BCExpn} can easily be extended to give an asymptotic expansion with an arbitrary number of terms.  Furthermore, the full asymptotic expansion for the modified Bessel function is also well-known (cf. (4.12.7) in \cite{AAR99}).

However, for our present purposes we need only the two terms in the theorem statement, along with the fact that as $x \to \infty$, we have
\begin{equation}
\label{E:IAsymp}
I_s(x) = \frac{e^x}{\sqrt{2\pi x}}\left(1 + O\left(\frac{1}{x}\right)\right).
\end{equation}
The asymptotic formulas in Theorem \ref{T:BCAsymp} then follow by applying \eqref{E:IAsymp} to the leading terms in in Theorem \ref{T:BCExpn}.  Since these terms are independent from $r$, the asymptotic differences of the enumeration functions are found by comparing the terms with $\beta_j(r-1)$ and $\beta_j(r)$, again using \eqref{E:IAsymp}.
\begin{corollary}
\label{C:Asymp}
As $n\rightarrow\infty$,
\begin{align*}
B_{d,r-1}(n) - B_{d,r}(n) & \sim
\frac{\pi}{2^{\frac{7}{4}}3^{\frac{3}{4}}d^{\frac{3}{4}}}
\left(\frac{1}{2} - \frac{r}{d} + \frac{1}{2d}\right)
n^{-\frac{5}{4}} e^{\pi\sqrt{\frac{2n}{3d}}}, \\
C_{d,r-1}(n) - C_{d,r}(n) & \sim
\frac{\pi}{2^{\frac{7}{4}}3^{\frac{7}{4}}d^{\frac{3}{4}}}
\left(\frac{1}{2} - \frac{r}{d} + \frac{1}{2d}\right)
n^{-\frac{5}{4}} e^{\pi\sqrt{\frac{2n}{3d}}}.
\end{align*}
\end{corollary}

In order to prove these results, we apply Cauchy's Theorem and  recover the coefficients from the generating functions.  Throughout this section we adopt the convenient shorthand notation
\begin{equation*}
F_1(q) := \sB_{d,r}(q) \qquad \text{and} \qquad F_2(q) := \sC_{d,r}(q).
\end{equation*}
Adopting similar shorthand for the coefficients, Cauchy's Theorem then implies that for any $n \geq 1$
\begin{equation}
\label{E:Cauchy}
c_1(n) := B_{d,r}(n)=\frac1{2\pi i}\int_{\calC}\frac{F_1(q)}{q^{n+1}}dq
\qquad \text{and} \qquad
c_2(n) := C_{d,r}(n)=\frac1{2\pi i}\int_{\calC}\frac{F_2(q)}{q^{n+1}}dq,
\end{equation}
with the contour $\calC$ chosen to be the (counter-clockwise) circle with radius $e^{-N}$, where we further define $N :=\frac{\pi}{\sqrt{6dn}}$.  It is convenient to parameterize this contour by setting $q = e^{-z}$, where $z = N + iy$ and $-\pi < y \leq \pi$. Note that we must be sure to use $\tau = \frac{iz}{2\pi}$ when applying automorphic transformations.

We then decompose the contour into $\calC = \calC_1 + \calC_2$, with
\begin{equation*}
\calC_1 := \Big\{ q \in \calC \: : \: |y| < 2N \Big\},
\end{equation*}
and $\calC_2$ consisting of the remaining curve.  We further denote the corresponding contributions to \eqref{E:Cauchy} for $j=1,2$ by
\begin{equation}
\label{E:MjEj}
M_j := \frac1{2\pi i}\int_{\calC_1}\frac{F_j(q)}{q^{n+1}}dq
\qquad \text{and} \qquad
E_j := \frac1{2\pi i}\int_{\calC_2}\frac{F_j(q)}{q^{n+1}}dq,
\end{equation}
as we will see that these integrals, respectively, contribute the main asymptotic term and error terms.

\subsection{Asymptotic behavior near $q=1$.}  We begin by determining the asymptotic behavior of the functions $F_j(q)$ on $\calC_1$, which contributes to the main terms for the coefficients.
\begin{lemma}
\label{L:FonC1}
If $q = e^{-z} \in \calC_1$ and $j = 1,2$, then we have the bounds
\begin{equation*}
F_{j}(q)=e^{\frac{\pi^2}{6dz}}\left(\alpha'_j +\beta'_j(r) z+O\left(z^2\right)\right),
\end{equation*}
where the constants are given by
\begin{eqnarray*}
\begin{array}{ll}\alpha'_1 = 1, & \displaystyle \alpha'_2 = \frac{1}{3}, \\
\displaystyle \beta'_1(r) = \frac{d}{12} - \frac{r}{2} + \frac{r^2}{2d}, \quad
&
\displaystyle \beta'_2 (r) = \frac{11d}{108}-\frac{r}{6}+\frac{r^2}{6d}.
\end{array}
\end{eqnarray*}
\end{lemma}
\begin{proof}
We begin with $F_1(q)$, recalling Proposition \ref{P:Bmod}.  Combined with the inversion formulas \eqref{E:etainv} and \eqref{E:thetainv}, this gives
\begin{equation*}
\sB_{d,r}(q) = -q^{-\frac{d}{12} + \frac{r}2} \frac{\vartheta \left( \frac12 + \frac{irz}{2\pi}; \frac{idz}{2\pi}\right)}{\eta \left(\frac{idz}{2\pi}\right)} =
-\frac{ie^{-\frac{\pi ir}{d}+z\left(\frac{d}{12}-\frac{r}{2}+\frac{r^2}{2d}\right)-\frac{\pi^2}{2dz}}\vartheta\left(-\frac{\pi i}{dz}+\frac{r}{d}; \frac{2\pi i}{dz}\right)}{\eta\left(\frac{2\pi i}{dz}\right)}.
\end{equation*}
Recalling \eqref{E:eta}, \eqref{E:thetaprod}, and \eqref{E:thetaneg},  we find a (uniform) bound for \eqref{E:BMaj} on $\calC_1$, namely
\begin{equation}
\label{E:BMaj}
\sB_{d,r}(q) = e^{\frac{\pi^2}{6dz}+z\left(\frac{d}{12}-\frac{r}{2}+\frac{r^2}{2d}\right)}\left(1+O\left(e^{-\frac{2\pi^2}{dz}}\right)\right).
\end{equation}

Turning next to $F_2(q)$, by \eqref{E:C=B} and \eqref{E:BMaj}, we find its asymptotic expansion by first directly calculating the Taylor expansion around $z=0$ for the convergent sum
\begin{equation}
\label{E:Gsum}
G(q):=\sum_{n\geq 0}\frac{q^{dn(n+1)}}{\left(-q^r, -q^{d-r}; q^d\right)_{n+1}}=G(0)+G'(0)z+O\left(z^2\right).
\end{equation}
The constant term evaluates to
\begin{equation*}
G(0)=\frac14\sum_{n\geq 0}\frac1{4^n}=\frac13.
\end{equation*}
To calculate the derivative, we use the fact that
\begin{equation*}
\frac{dG}{dz} = \frac{dG}{dq} \frac{dq}{dz} = -q \frac{dG}{dq}
\end{equation*}
and then apply logarithmic differentiation to each summand in order to find the evaluation
\begin{align*}
G'(0)&=-\frac14\sum_{n\geq 0}\frac1{4^n}
\left(dn^2+dn - \sum_{j=0}^n
\left(\frac{r+dj}{2}+\frac{d+dj-r}{2}\right)\right)
=-\frac{d}{8}\sum_{n\geq 0}\frac{\left(n^2-1\right)}{4^n}=\frac{2d}{27}.
\end{align*}
Combining \eqref{E:C=B}, \eqref{E:BMaj}, and \eqref{E:Gsum} then gives the claim.
\end{proof}

\subsection{Asymptotic behavior away from $q=1$}
We next determine the asymptotic behavior of the $F_j(q)$ on $\calC_2$.  It is sufficient to find asymptotic bounds, as this contour only contributes to the error term in the overall formulas for the coefficients.
\begin{lemma}
\label{L:FonC2}
If $q \in \mathcal{C}_2$, then the following bounds are satisfied:
\begin{enumerate}
\item
$\displaystyle F_1(q) \ll e^{\frac{\pi \sqrt{n}}{5\sqrt{6d}}},$
\item
$\displaystyle F_2(q) \ll n e^{\frac{\pi \sqrt{2n}}{5\sqrt{3d}}}.$
\end{enumerate}
\end{lemma}
\begin{proof}
We begin by observing that the inversion formulas used in proving \eqref{E:BMaj} also immediately lead to a bound for $F_1(q)$ on $\calC_2$.  Namely, we have
\begin{equation*}
\sB_{d,r}(q) \ll e^{\frac{\pi^2}{6d} \re \left( \frac{1}{z} \right)}.
\end{equation*}

In order to bound $F_2(q)$, we recall Proposition \ref{P:Chyper} and estimate the additional pieces individually.  We first address the infinite product, again using \eqref{E:etainv} to conclude that
\begin{equation*}
\frac{1}{\left( q^d ; q^d \right)_\infty} \ll |z|^{\frac12} e^{\frac{\pi^2}{6d} \re \left( \frac{1}{z} \right)}.
\end{equation*}
It remains to bound the sum.  We use the fact that $\re(z) = N$ on $\calC_2$ and calculate the (rough) bound
\begin{equation*}
\sum_{n\in\Z}(-1)^n \frac{q^{\frac{3dn(n+r)}{2}}}{1+q^{dn+r}} \ll\frac{1}{1-e^{-N}}\sum_{n\geq 0} e^{-nN} \ll \frac1{N^2}\ll n.
\end{equation*}
Thus
\begin{align*}
G_2(q) & \ll |z|^\frac12 n e^{\frac{\pi^2}{3d} \re\left( \frac{1}{z} \right)}.
\end{align*}
The statement now follows because $|z| \ll 1$ on all of $\calC$, and furthermore, for $q \in \mathcal{C}_2$ we have the additional inequality
\begin{equation*}
\re \left( \frac{1}{z} \right) = \frac{\re (z)}{\re (z)^2 + \im(z)^2} \leq \frac{1}{5N}.
\end{equation*}
\end{proof}

\subsection{Asymptotic formulas for coefficients}
We now complete the proofs of Theorem \ref{T:BCExpn} and Corollary \ref{C:Asymp} by plugging the bounds for the $F_j(q)$ into \eqref{E:MjEj}.  We begin by considering the first two terms from Lemma \ref{L:FonC1}, and we relate the corresponding integrals to Bessel functions.
In particular, Wright's calculations in Section 5 of \cite{Wri71} apply directly to the present situation, implying that
\begin{align}
\label{E:MainBessel}
\frac{1}{2\pi i} & \int_{\mathcal{C}_1} \frac{e^{\frac{\pi^2}{6dz}} \left( \alpha_j' + \beta_j' (r) z \right)}{q^{n+1}} dq \\
& =\alpha_j'\left(\frac{\pi}{\sqrt{6d}}\right) n^{-\frac12} I_{-1}\left(\pi\sqrt{\frac{2n}{3d}}\right)
+\beta_j'(r)\frac{\pi^2}{6d}n^{-1}I_{-2}\left(\pi\sqrt{\frac{2n}{3d}}\right)
+O\left(n^{-1} e^{\frac{\pi}{2}\sqrt{\frac{3n}{2d}}}\right). \notag
\end{align}

We now turn to the error terms.  Using the fact that $|z| \leq \sqrt{5} N$ on $\calC_1$, we find that the error terms from Lemma \ref{L:FonC1} for either $j=1,2$ contribute
\begin{equation}
\label{E:C1Err}
\int_{\mathcal{C}_1} e^{nN+\frac{\pi^2}{6d}\text{Re}\left(\frac1z\right)}|z|^2 dz\ll N^3 e^{\frac{\pi\sqrt{2n}}{\sqrt{3d}}}\ll n^{-\frac32} e^{\frac{\pi\sqrt{2n}}{\sqrt{3d}}}.
\end{equation}
The bounds from Lemma \ref{L:FonC2} on $\calC_2$ give a contribution with an exponentially lower order, so the overall error is given by \eqref{E:C1Err}.  Inserting \eqref{E:MainBessel} and \eqref{E:C1Err} into \eqref{E:MjEj}, we then obtain Theorem \ref{T:BCExpn}.

\section{Probabilistic interpretation of universal mock theta functions}
\label{S:Prob}

In this section we further examine the combinatorial properties of Schur-type partitions and consequently prove the amazing fact that the universal mock theta function at real arguments naturally occurs as the conditional probability of events in simple probability spaces.  This phenomenon was previously observed for individual examples of Ramanujan's mock theta functions (see \cite{Wat36} for notation), including $\xi(q)$ \cite{AEPR07} and $\phi(q)$ \cite{BMM13}.  However, our current results are significantly more fundamental due to the underlying importance of the universal mock theta function.

Suppose that $0 < q < 1$ is fixed, and let $E_1, E_2, \dots$ be a sequence of independent events that individually occur with probabilities
\begin{equation}
\label{E:ProbE}
p_j = \Prob(E_j) := \frac{q^j}{1+q^j}.
\end{equation}
We also denote the complementary events by $F_j := E_j^c$, which have corresponding probabilities $\o{p}_j := \Prob(F_j) = 1 - p_j = \frac{1}{1+q^j}.$
For any events $R$ and $S$, we adopt the space-saving notational conventions $RS := R \cap S.$

If $d \geq 3$ and $1 \leq r < \frac{d}{2}$, we consider certain events defined in terms of the sequence of $E_j$s, although we first introduce one more notational shorthand, writing $E_n^k := E_{nd+k}$ (with similar notation for the complementary $F$s).  We now define the events
\begin{align}
U_{d,r} & := \bigcap_{n \geq 0} \Big(E_{n}^r F_{n}^{d-r} F_{n+1}^0 \cup F_{n}^r \Big)
\Big(E_{n}^{d-r} F_{n+1}^{0} F_{n+1}^r \cup F_{n}^{d-r} \Big)
\Big(E_{n+1}^{0} F_{n+1}^{r} F_{n+1}^{d-r} F_{n+2}^0 \cup F_{n+1}^{0} \Big), \notag \\
\label{E:UV}
V_{d,r} & := \bigcap_{n \geq 1} \Big(E_{n}^r F_{n}^{d-r} F_{n+1}^0 \cup F_{n}^r \Big)
\Big(E_{n}^{d-r} F_{n+1}^{0} F_{n+1}^r \cup F_{n}^{d-r} \Big)
\Big(E_{n+1}^{0} F_{n+1}^{r} F_{n+1}^{d-r} F_{n+2}^0 \cup F_{n+1}^{0} \Big).
\end{align}
In words, $U_{d,r}$ is the event such that if $E_{nd+r}$ occurs, then $E_{nd+d-r}$ and $E_{(n+1)d}$ do not occur; if $E_{nd +d-r}$ occurs, then $E_{(n+1)d}$ and $E_{(n+1)d+r}$ do not occur; and if $E_{(n+1)d}$ occurs, then $E_{(n+1)d+r}$, $E_{(n+1)d+d-r}$ and $E_{(n+2)d}$ do not occur.  The event $V_{d,r}$ has the same conditions beginning only from $E_{d+r}$, with no restrictions on whether $E_{r}, E_{d-r}$, and $E_d$ occur.
Note that the events $U_{d,r}$ and $V_{d,r}$ are independent from any $E_j$ with $j \not\equiv 0, \pm r \pmod{d}.$

\begin{theorem}
\label{T:prob}
Suppose that $0 < q < 1$, $d \geq 3$ and $1 \leq r < \frac{d}{2}.$  The following identities hold:
\begin{enumerate}
\item
\label{T:prob:U|V}
$\displaystyle \Prob(U_{d,r} \mid V_{d,r}) = \frac{1}{\left(1+q^r\right) \left(1+q^{d-r}\right) \left(1+q^d\right)} \cdot \frac{1}{g_3\left(-q^r; q^d\right)}$,
\item
\label{T:prob:F|U}
$\displaystyle \Prob(F_r F_{d-r} F_{d} \mid U_{d,r}) = g_3\left(-q^r; q^d\right).$
\end{enumerate}
\end{theorem}
\begin{proof}
Let $\calV_k$ denote the event that all of the conditions in the definition of $U_{d,r}$ are met beginning from $E_{k}$, with no restrictions on $E_j$ for $j < k$.  For example, $\calV_r = U_{d,r}$, and $\calV_{d+r} = V_{d,r}.$
The following three recurrences follow from the definition of $\calV_j$ in terms of the gap conditions in $U_{d,r}$:
\begin{align}
\Prob(\calV_{kd}) &= p_{kd} \o{p}_{kd+r} \o{p}_{kd+d-r} \o{p}_{(k+1)d} \Prob(\calV_{(k+1)d+r})
+ \o{p}_{kd} \Prob(\calV_{kd+r}), \notag \\
\label{E:Vkd3rec}
\Prob(\calV_{kd+r}) &= p_{kd+r}\o{p}_{kd+d-r} \o{p}_{(k+1)d} \Prob(\calV_{(k+1)d+r})
+ \o{p}_{kd+r} \Prob(\calV_{kd+d-r}), \\
\Prob(\calV_{kd+d-r}) &= p_{kd+d-r} \o{p}_{(k+1)d} \o{p}_{(k+1)d+r} \Prob(\calV_{(k+1)d+d-r})
+ \o{p}_{kd+d-r} \Prob(\calV_{(k+1)d}). \notag
\end{align}

We now combine the three recurrences in \eqref{E:Vkd3rec} into one.  Note that the first recurrence already expresses $\Prob(\calV_{kd})$ in terms of $\Prob(\calV_{jd+r})$ for various $j$, and we can use the second recurrence to do the same for $\Prob(\calV_{kd+d-r})$, finding
\begin{equation*}
\Prob(\calV_{kd+d-r}) = \frac{1}{1-p_{kd+r}} \bigg(\Prob(\calV_{kd+r}) - p_{kd+r}(1-p_{kd+d-r})(1-p_{(k+1)d})\Prob(\calV_{(k+1)d+r})\bigg).
\end{equation*}
Plugging in this formula and the first line of \eqref{E:Vkd3rec} into the third line then gives an identity involving only $\calV_{kd+r}, \calV_{(k+1)d+r},$ and  $\calV_{(k+2)d+r}$.  The resulting expression simplifies to the following recurrence, where we write $\calU_j := \calV_{jd+r}$ to save space
\begin{align}
\label{E:UkRec}
\Prob(\calU_k) & = \Big(p_{kd+r} \o{p}_{kd+d-r} \o{p}_{(k+1)d} +
\o{p}_{kd+r}p_{kd+d-r}\o{p}_{(k+1)d}
+ \o{p}_{kd+r}\o{p}_{kd+d-r}\o{p}_{(k+1)d}\Big) \Prob(\calU_{k+1}) \\
& \qquad \quad + \Big(\o{p}_{kd+r} \o{p}_{kd+d-r} p_{(k+1)d} \o{p}_{(k+1)d+r} \o{p}_{(k+1)+d-r} \o{p}_{(k+2)d} \notag \\
& \qquad \qquad \qquad - \o{p}_{kd+r} p_{kd+d-r} \o{p}_{(k+1)d} p_{(k+1)d+r} \o{p}_{(k+1)+d-r} \o{p}_{(k+2)d}\Big)
\Prob(\calU_{k+2}). \notag
\end{align}

We note that this is analogous to Andrews' proof of Theorem 1 in \cite{And68}, where he derives three recurrences for the Schur-type partitions enumerated by $C_{3,1}(n)$ based on their smallest part.  Furthermore, just as in our proof of Proposition \ref{P:frec}, one can alternatively show \eqref{E:UkRec} directly by by conditioning on whether $E_{kd+r}, E_{kd+d-r}$ or $E_{(k+1)d}$ occur, and then subtracting off the disallowed sequence $F_{kd+r} E_{kd+d-r} F_{(k+1)d} E_{(k+1)d + r} F_{(k+1)d + d-r} F_{(k+2)d}$.

In order to more thoroughly describe the relationship between the events $\calV_j$ and Schur-type partition functions, we renormalize the generating function by defining
\begin{equation}
\label{E:hdr}
h_{d,r}(x) = h_{d,r}(x;q) := \frac{f_{d,r}(x)}{\left(-xq^r, -xq^{d-r}, -xq^d; q^d\right)_\infty}.
\end{equation}
Proposition \ref{P:frec} then becomes
\begin{align}
\label{E:hrec}
h_{d,r}(x) = & \frac{1 + xq^r + xq^{d-r}}{\left(1+xq^r\right) \left(1+xq^{d-r}\right) \left(1+xq^d\right)}
h_{d,r}\left(xq^d\right) \\
& + \frac{xq^d - x^2q^{2d}}{\left(1+xq^r\right) \left(1+xq^{d-r}\right) \left(1+xq^d\right) \left(1+xq^{d+r}\right) \left(1+xq^{2d-r}\right) \left(1+xq^{2d}\right)}h_{d,r}\left(xq^{2d}\right).
\notag
\end{align}

If we now define $\calH_k = \calH_k(q) := h_{d,r}(q^{kd})$ and recall \eqref{E:ProbE}, then \eqref{E:hrec} implies that the recurrence \eqref{E:UkRec} holds with $\calH_k$ in place of $\Prob(\calU_k).$
We observe that as $k \to \infty$, we have the limit $\calH_k \to 1$, because $h_{d,r}(x) \to 1$ as $x \to 0$.  Similarly, we also have $\Prob(\calU_k) \to 1$ since there are no conditions on any $E_j$ in the limit.  This boundary condition guarantees that the recurrence has a unique solution (cf. the theory of $q$-difference equations \cite{And68Q}), hence
\begin{equation*}
\label{E:U=h}
\Prob(\calU_k) = \calH_k(q) = h_{d,r}\left(q^{kd}\right).
\end{equation*}

We can now complete the proof of the theorem.  For part \ref{T:prob:U|V}, we calculate
\begin{equation}
\label{E:U|V}
\Prob(U_{d,r} \mid V_{d,r}) = \frac{\Prob(U_{d,r})}{\Prob(V_{d,r})} = \frac{\Prob(\calU_0)}{\Prob(\calU_1)}
= \frac{f_{d, r}(1)}{\left(1+q^r\right)\left(1+q^{d-r}\right)\left(1+q^d\right) f_{d, r}\left(q^d\right)},
\end{equation}
where the last equality is due to \eqref{E:hdr}.
The theorem statement then follows from \eqref{E:Schur}, \eqref{E:BC=f} and \eqref{E:hdr}, which together imply that $f_{d,r}(q^d) = f_{d,r}(1) \cdot g_3(-q^r; q^d).$

For part \ref{T:prob:F|U}, we similarly have
\begin{equation*}
\Prob(F_r F_{d-r} F_d \mid U_{d,r}) = \frac{\Prob(F_r F_{d-r} F_d \cap U_{d,r})}{\Prob(U_{d,r})}
= \frac{\Prob(F_r F_{d-r} F_d) \Prob(V_{d,r})}{\Prob(U_{d,r})}
= g_3\left(-q^r; q^d\right),
\end{equation*}
where the final equality follows from \eqref{E:ProbE} and the inverse of \eqref{E:U|V}.
\end{proof}
\begin{remark*}
Just as we worked directly with the combinatorics of $q$-difference equations in order to prove Proposition \ref{P:fqseries} (thereby providing short new proofs of \eqref{E:Schur} and \eqref{E:AndrewsC3}), our use of probability arguments above could also be adapted to give new proofs of the results relating generating functions and probability found in Section 4 of \cite{HLR04} and Section 6 \cite{BMM13} (particularly (6.1)).
\end{remark*}


\bigskip

\begin{thebibliography}{99}


\bibitem{Ald48} H. Alder, {\it The nonexistence of certain identities in the theory of partitions and compositions}, Bull. Amer. Math. Soc. {\bf 54} (1948), 712--722.

\bibitem{AG93} K. Alladi and B. Gordon, {\it Generalizations of Schur's partition theorem}, Manuscripta Math. {\bf 79} (1993), 113--126.

\bibitem{Ald69} H. Alder, {\it Proof of Andrews' conjecture on partition identities}, Proc. Amer. Math. Soc. {\bf 22} (1969), 688--689.

\bibitem{AJL11} C. Alfes, M. Jameson, and R. Lemke Oliver, {\it Proof of the Alder-Andrews Conjecture}, Proc. Amer. Math. Soc. {\bf 139} (2011), 63--78.

\bibitem{And66} G. Andrews, {\it An analytic proof of the Rogers-Ramanujan-Gordon identities}, Amer. J. Math. {\bf 88} (1966), 844--846.

\bibitem{And68} G. Andrews, {\it On partition functions related to Schur's second partition theorem}, Proc. Amer. Math. Soc. {\bf 19} (1968), 441--444.

\bibitem{And68Q} G. Andrews, {\it On $q$-difference equations for certain well-poised basic hypergeometric series}, Q. J. Math. {\bf 19} (1968), 433--447.

\bibitem{And69} G. Andrews, {\it A general theorem on partitions with difference conditions},
Amer. J. Math. {\bf 91} (1969), 18--24.

\bibitem{And71P} G. Andrews, {\it On a partition problem of H. L. Alder,} Pacific J. Math. {\bf 36} (1971), 279--284.

\bibitem{And75} G. Andrews, {\it Problems and prospects for basic hypergeometric functions}, Theory and application of special functions, pp. 191--224. Math. Res. Center, Univ. Wisconsin, Publ. No. 35, Academic Press, New York, 1975.

\bibitem{And81} G. Andrews, {\it The hard-hexagon model and Rogers-Ramanujan type identities}, Proc. Nat. Acad. of Sci. {\bf 78} (1981), 5290--5292.

\bibitem{And98} G. Andrews,   \emph{The theory of partitions},
Cambridge University Press, Cambridge, 1998.

\bibitem{And07} G.  Andrews, \emph{Partitions, Durfee symbols, and the
 Atkin-Garvan moments of ranks},
 Invent. Math. {\bf 169} (2007), 37--73.

\bibitem{AAR99} G.   Andrews, R. Askey, and R.  Roy, \emph{Special functions}, Encyclopedia of Mathematics and its Applications {\bf 71}, Cambridge University Press, Cambridge, 1999.

\bibitem{AEPR07} G. Andrews, H. Eriksson, F. Petrov, and D. Romik,  {\it Integrals, partitions and MacMahon's theorem}, J. Combin. Theory Ser. A {\bf 114} (2007), 545--554.

\bibitem{ARZ13} G. Andrews, R. Rhoades and S. Zwegers, {\it Modularity of the concave composition generating function}, to appear in Alg. and Number Thy.

\bibitem{Apo90} T. Apostol, \emph{Modular Functions and Dirichlet Series in Number Theory
Series: Graduate Texts in Mathematics, Vol. 41},
2nd ed., 1990.

\bibitem{Bre80} D. Bressoud, {\it A combinatorial proof of Schur's 1926 partition theorem}, Proc. Amer. Math. Soc. {\bf 79} (1980), 338--340.

\bibitem{BruF04} J. Bruinier and J. Funke, {\it On two geometric theta lifts}, Duke
Math. J. {\bf 125} (2004), 45--90.

\bibitem{BM11} K. Bringmann and K. Mahlburg, \emph{An extension of the Hardy-Ramanujan Circle Method and  applications to partitions without sequences}, Amer. Journal of Math {\bf 133} (2011), 1151--1178.

\bibitem{BM12} K. Bringmann and K. Mahlburg, {\it Asymptotic inequalities for positive crank and rank moments,} to appear in Trans. Amer. Math. Soc.

\bibitem{BMM13} K. Bringmann, K. Mahlburg, and A. Mellit,
{\it Convolution Bootstrap Percolation Models, Markov-type Stochastic Processes, and Mock Theta Functions},
Int. Math. Res. Not. (2013), Vol. 2013, 971--1013.

\bibitem{BMan13} K. Bringmann and J. Manschot, {\it Asymptotic formulas for coefficients of inverse theta functions}, preprint, \texttt{arXiv:1304.7208}.

\bibitem{BO09} K. Bringmann and K. Ono, {\it Some characters of Kac and Wakimoto and nonholomorphic modular functions}, Math. Ann. {\bf 345} (2009), 547--558.

\bibitem{CDH13} M. Cheng, J. Duncan, and J. Harvey, {\it Umbral moonshine}, preprint, \texttt{arXiv:1204.2779}.

\bibitem{DMZ12} A. Dabholkar, S. Murthy, and D. Zagier, {\it Quantum black holes, wall crossing, and mock modular forms}, preprint, \texttt{arXiv:1208.4074 [hep-th]}.

\bibitem{EZ85} M. Eichler, and D. Zagier, {\it The theory of Jacobi forms}, Progress in Math. 55, Birkh\"auser Boston, MA, 1985.

\bibitem{GR90} G. Gasper and M. Rahman, {\it Basic hypergeometric series}, Encycl. of Math. and Applications
{\bf 35}, Cambridge University Press, Cambridge, 1990.

\bibitem{Gle28} W. Gleissberg, {\it \"{U}ber einen Satz von Herrn I. Schur},
Math. Z. {\bf 28} (1928), 372--382.

\bibitem{Gor61} B. Gordon, {\it A combinatorial generalization of the Rogers-Ramanujan identities},
Amer. J. Math. {\bf 83} (1961), 393--399.

\bibitem{Gor65} B. Gordon, {\it Some continued fractions of the Rogers-Ramanujan type},
Duke Math. J. {\bf 32} (1965), 741--748.

\bibitem{GM12} B. Gordon and R. McIntosh, {\it A survey of the classical mock theta functions}, Partitions, $q$-series, and modular forms, Dev. Math. {\bf 23}, Springer, New York, 2012, 95--244.

\bibitem{Hic88} D. Hickerson, {\it On the seventh order mock theta functions}, Invent. Math. {\bf 94} (1988), 661--677.

\bibitem{HLR04} A. Holroyd, T. Liggett, and D. Romik, {\it Integrals, Partitions, and Cellular Automata}, Trans.  Amer. Math. Soc.  {\bf 356} (2004), 3349--3368.

\bibitem{Kan09} S. Kang, {\it Mock Jacobi forms in basic hypergeometric series}, Compos. Math. {\bf 145} (2009), 553--565.

\bibitem{Kob84} N. Koblitz, \emph{Introduction to elliptic curves and modular forms},
 Graduate Texts in Mathematics {\bf 97}, Springer-Verlag, New York, 1984.

\bibitem{LW81} J. Lepowsky and R. Wilson, {\it A new family of algebras underlying the Rogers-Ramanujan identities and generalizations}, Proc. Nat. Acad. Sci. {\bf 78} (1981), 7254--7258.

\bibitem{Rad43} H. Rademacher, \emph{On the Expansion of the Partition Function in a Series},
Ann. of Math. {\bf 44} (1943), 416--422.

\bibitem{RR19} L. Rogers and S. Ramanujan, {\it Proof of certain identities in combinatory analysis}, Math. Proc. Cambridge Philos. Soc. {\bf 19} (1919), 211--216.

\bibitem{Sch26} I. Schur, {\it Zur additiven Zahlentheorie}, Sitzungsber. Preuss. Akad. Wiss. Phys.-Math. Kl., 1926.

\bibitem{Stem90} J. Stembridge, {\it Hall-Littlewood functions, plane partitions, and the Rogers-Ramanujan identities},
Trans. Amer. Math. Soc. {\bf 319} (1990), 469--498.

\bibitem{Wat36} G. Watson, {\it The final problem: An account of the mock theta functions}, J. Lond. Math. Soc. {\bf 11} (1936), 55--80.

\bibitem{Wri41} E. Wright, \emph{Asymptotic partition formulae II. Weighted partitions}, Proc. Lond. Math. Soc. (2) {\bf 36} (1933), 117--141.

\bibitem{Wri71} E. Wright, \emph{Stacks. II,} Q. J. Math. Ser. (2) {\bf 22} (1971), 107--116.

\bibitem{Yee04} A. Yee, {\it Partitions with difference conditions and Alder's conjecture}, Proc. Natl. Acad. Sci. {\bf 101} (2004),
16417--16418.

\bibitem{Yee08} A. Yee, {\it Alder's conjecture,} J. Reine Angew. Math. {\bf 616} (2008), 67--88.


\bibitem{Zag06} D. Zagier, {\it The dilogarithm function},
Frontiers in Number Theory, Physics and Geometry II, Springer-Verlag, Berlin-Heidelberg-New York (2006), pp. 3--65.

\bibitem{Zag09} D. Zagier, \emph{Ramanujan's mock theta functions and their
 applications [d'apr\'es  Zwegers and Bringmann-Ono] }
Ast\'erisque {\bf 326} (2009), Soc. Math. de France, 143--164.

\bibitem{Zw02} S. Zwegers, \emph{Mock theta functions},
Ph.D. Thesis, Universiteit Utrecht, 2002.

\end{thebibliography}
\end{document}